\documentclass[11pt]{amsart}
\input{macros}
\UseRawInputEncoding
\title[Kaufmann--Clote question and weak regularity principle]{The Kaufmann–-Clote question on end extensions of models of arithmetic and the weak regularity principle}

\author{Mengzhou Sun}
\address{Department of Mathematics, National University of Singapore, Singapore 119076}
\email{sunm07@u.nus.edu}
\begin{document}

\begin{abstract}
    We investigate the end extendibility of models of arithmetic with restricted elementarity. By utilizing the restricted ultrapower construction in the second-order context, for each $n\in\IN$ and any countable model of $\bd\Sigma_{n+2}$, we construct a proper $\Sigma_{n+2}$-elementary end extension satisfying $\bd\Sigma_{n+1}$, which answers a question by Clote positively. 
    We also give a characterization of countable models of $\ind\Sigma_{n+2}$ in terms of their end extendibility similar to the case of $\bd\Sigma_{n+2}$. 
    Along the proof, we will introduce a new type of regularity principles in arithmetic called the weak regularity principle, which serves as a bridge between the model's end extendibility and the amount of induction or collection it satisfies.
\end{abstract}
\maketitle
\section{Introduction}
End extensions are of significant importance and have been studied intensively in the model theory of arithmetic: The classical MacDowell--Specker theorem~\cite{incoll:MDS} showed that every model of $\PA$ admits a proper elementary end extension. Around two decades later, Paris and Kirby~\cite{incoll:PK} studied the hierarchical version of the MacDowell--Specker theorem for fragments of $\PA$. In fact, they showed that for countable models, end extendibility with elementarity characterizes the collection strength of the ground model.
\begin{thm}[Paris--Kirby]\label{thm:paris-kirby}
    Let $M$ be a countable model of $\ind\Delta_0$. For each $n\in\IN$, $M$ satisfies $\bd\Sigma_{n+2}$ if and only if $M$ has a proper $\Sigma_{n+2}$-elementary end extension $K$.
\end{thm}
For the left-to-right direction, the theorem above does not explicitly specify what theory the end extension $K$ can satisfy.
This amount of elementarity stated in the theorem already implies $K\models \ind \Sigma_n$, and Paris--Kirby's proof actually indicates that $K$ cannot always satisfy $\ind\Sigma_{n+1}$, since this would imply $M\models \bd\Sigma_{n+3}$. Moreover,  for each $n\in\IN$, Cornaros and Dimitracopoulos~\cite{art:cd} constructed a countable model of $\bd\Sigma_{n+2}$ which does not $\Sigma_{n+1}$-elementarily end extend to any model of $\ind\Sigma_{n+1}$.
So with $\Sigma_{n+2}$-elementarity, the theory that the end extension $K$ may satisfy lies between $\ind\Sigma_n$ and $\ind\Sigma_{n+1}$, and the following question arises naturally:
\begin{question}[Kaufmann--Clote]\label{question:kc}
    For $n\in\IN$, does every countable model $M\models\bd\Sigma_{n+2}$ have a proper $\Sigma_{n+2}$-elementary end extension $K\models \bd\Sigma_{n+1}$?
\end{question}
The question was included in the list of open problems in~\cite[Page 12, Problem 33]{incoll:openproblems} edited by Clote and Kraj\'i\v cek.
It was first raised by Clote in \cite{incoll:clote}, where he noted that the same question in the context of models of set theory had been previously posed by Kaufmann in~\cite{incoll:Kaufmann}.
In the same paper, Clote~\cite[Proposition 7]{incoll:clote} showed that every countable model $M\models\ind\Sigma_{n+2}$ admits a $\Sigma_{n+2}$-elementary proper end extension to some $K\models M\text{-}\bd\Sigma_{n+1}$, which means all the instances of $\bd\Sigma_{n+1}$:
\[\falt x a\ex y\phi(x,y)\rightarrow\ex b\falt x a\exlt y b\phi(x,y)\]
where $a\in M$ while parameters in $K$ are allowed in $\phi(x,y)\in\Sigma_{n+1}(K)$.
Cornaros and Dimitracopoulos~\cite{art:cd} showed that every countable model $M\models \bd\Sigma_{n+2}$ has a $\Sigma_{n+2}$-elementary proper end extension $K\models \bd\Sigma_{n+1}^-$(the parameter-free $\Sigma_{n+1}$-collection).
In this paper, we give an affirmative answer to Question~\ref{question:kc}.

The original proof of Paris--Kirby's theorem is based on a first-order restricted ultrapower construction. The ultrapower is generated by a single element when viewed from the ground model. One can show that, by relativizing the proof of pointwise $\Sigma_{n+1}$-definable models do not satisfy $\bd\Sigma_{n+1}$ (e.g.,~\cite[IV, Lemma 1.41]{book:HP}, such ultrapowers always fail to satisfy $\bd\Sigma_{n+1}$ in the question.
To tackle this issue, we expand our model into a second-order structure satisfying $\WKL_0$ and the end extension $K$ will be a second-order restricted ultrapower with respect to that structure.

For us, one of the motivations of studying this question is to find a model-theoretic characterization of (countable) models of $\ind\Sigma_{n+2}$ analogous to Theorem~\ref{thm:paris-kirby}.
Despite the fact that the extension in Question~\ref{question:kc} is insufficient for characterization, a slight generalization of it will suffice.
We will show that for any countable model $M\models\ind\Delta_0+\exp$,
$M\models\ind\Sigma_{n+2}$ if and only if $M$ admits a proper $\Sigma_{n+2}$-elementary end extension $K\models M\hyp\ind\Sigma_{n+1}$, whose definition is similar to $M\hyp\bd\Sigma_{n+1}$.

The regularity principle is the key to connecting end extensions with the arithmetic theories that the models satisfy.
Through the end extension, we can employ a `nonstandard analysis' style argument to prove certain types of regularity principle in the ground model.
Here we provide an example of such an argument. Notice that $K\models\bd\Sigma_{n+1}$ plays an crucial role in the proof. 
\begin{proposition}\label{prop:endextrpi}
    For each $n\in \IN$, let $M\models\ind\Delta_0$. If $M$ admits a proper $\Sigma_{n+2}$-elementary end extension $K\models \bd\Sigma_{n+1}$, then $M$ satisfies the following principle:
    \[\fa x\exlt y a\phi(x,y)\rightarrow\exlt y a\excf x\phi(x,y).\]
    where $a\in M $, $\phi(x,y)\in \Pi_{n+1}(M)$ and $\exists^{\cf} x$ abbreviates $\fa b\exists x>b$. 
\end{proposition}
\begin{proof}
    Suppose $M\models\fa x\exlt y a\phi(x,y)$, then it is equivalent to a $\Pi_{n+1}$-formula over $\bd\Sigma_{n+1}$. Since both $M$ and $K$ are models of $\bd\Sigma_{n+1}$,
    $K$ satisfies the same formula by elementarity.
    Pick some arbitrary $d>M$ in $K$ and let $c<a$ such that $K\models\phi(d,c)$. Now for each $b\in M$, $K\models \exlg x b\phi(x,c)$ and it is witnessed by $d$. Transferring each of these formulas back to $M$ by elementarity, we have $M\models\exlg x b\phi(x,c)$ for any $b\in M$, which means $M\models\excf x\phi(x,c)$.
\end{proof}
We call the principle above \defm{weak regularity principle}, and denote it by $\wrd\phi$.
The proposition above, together with the affirmative answer to Question~\ref{question:kc}, implies that $\bd\Sigma_{n+2}\vdash \wrd\Pi_{n+1}$ for each $n\in \IN$.\par


Similar to the argument above, we will show that if the extension $K\models M\hyp\ind\Sigma_{n+1}$, then the ground model $M$ will satisfy some other form of weak regularity principle that implies $\ind\Sigma_{n+2}$.

The paper is organized as follows: 
In Section 2, we present the necessary notations and fundamental facts regarding models of arithmetic
In Section 3 we review the definition of the second-order restricted ultrapower, and state some basic properties of it. 
In Section 4, we provide an affirmative answer to Question~\ref{question:kc} and present the construction of an end extension that characterizes countable models of $\ind\Sigma_{n+2}$ as mentioned above.
In Section 5, we formally introduce the weak regularity principle $\wrd\phi$, and calibrate its strength within the I-B hierarchy.
Finally, putting the results in Section 4 and Section 5 together, we establish a model-theoretic characterization of countable models of $\ind\Sigma_{n+2}$ analogous to Theorem~\ref{thm:paris-kirby}.
\section{Preliminaries}
We assume the reader is familiar with some basic concepts and facts in model theory of first- and second-order arithmetic \cite{book:HP,book:kaye}. We reserve the symbol $\IN$ for the set of standard natural numbers. 
For each $n\in \IN$, let $\Sigma_n$ and $\Pi_n$ be the usual classes of formulas in the arithmetic hierarchy of first-order arithmetic.
Given a model of first-order arithmetic $M$, a formula is $\Delta_n$ over $M$ if it is equivalent to both a $\Sigma_n$ and a $\Pi_n$ formula in $M$, or simply $\Delta_n$ if the model involved is clear from the context.
$\Sigma_n\wedge\Pi_n$ is the class of formulas which is the conjunction of a $\Sigma_n$- and a $\Pi_n$-formula, and $\Sigma_n\vee\Pi_n$ is defined similarly.
$\Sigma_0(\Sigma_n)$ is the closure of $\Sigma_n$ formulas under Boolean operations and bounded quantification.
$\Sigma_n^0$, $\Pi_n^0$, $\Delta_n^0$ and $\Sigma_0(\Sigma_n^0)$ are their second-order variants respectively.
Given a model of first-order arithmetic $M$, $\Delta_n(M)$ is the class of $\Delta_n$-formulas over $M$, potentially including parameters from 
$M$ that are not explicitly shown.
$\Sigma_n(M)$, $\Pi_n(M)$ are defined similarly, as well as their second-order variants $\Delta_n^0(M,\mathcal{X})$, $\Sigma_n^0(M,\mathcal{X})$ and $\Pi_n^0(M,\mathcal{X})$ for some model of second-order arithmetic $(M,\mathcal{X})$.
Finally, $\excf {x}\dots$ is the abbreviation of $\fa b\exlg{x}{b}\dots$.\par
For each $n\in \IN$, let $\bd\Sigma_n$ and $\ind\Sigma_n$ be the collection scheme and the induction scheme for $\Sigma_n$ formulas respectively. We assume that all the $\bd\Sigma_n$ include $\ind\Delta_0$, and all the theories considered include $\PA^-$, which is the theory of non-negative parts of discretely ordered rings. $\ind\Sigma_n^0$ and $\bd\Sigma_n^0$ are their second-order counterparts, respectively.

For any element $c$ in some model $M\models\ind\Delta_0+\exp$, we identify $c$ with a subset of $M$ by defining $x\in c$ to mean the $x$-th digit in the binary expansion of~$c$ is~$1$.

$\RCA_0$ is the subsystem of second-order arithmetic consisting of Robinson arithmetic, $\ind\Sigma_1^0$, $\Delta_1^0$-comprehension, and $\WKL_0$ consists of $\RCA_0$ and a statement asserting that every infinite binary tree has an infinite path.
A tree $T$ in the second-order universe is called \defm{bounded} if there is a total function $f$ in the second-order universe such that $\sigma(x)<f(x)$ for any $\sigma\in T$ and $x<\len \sigma$. 
It is provable in $\WKL_0$ that every infinite bounded tree has an infinite path~\cite[Lemma IV.1.4]{book:simpson}. For each $n\in \IN$, every countable model of $\bd\Sigma_{n+2}^0$ admits a countable $\omega$-extension (i.e., an extension only adding second-order objects) to some model satisfying $\WKL_0+\bd\Sigma_{n+2}^0$~\cite{incoll:Hajek}.

Considering extensions of models of arithmetic, we say an extension $M\subseteq K$ of models of first-order arithmetic is \defm{$\Sigma_n$-elementary}, if all the $\Sigma_n(M)$-formulas are absolute between $M$ and $K$, and we write $M\elemsub_{\Sigma_n} K$.
We say an extension $M\subseteq K$ is an \defm{end extension}, if every element of $K\setminus M$ is greater than any element of the ground model $M$.
This is denoted by $M\subseteq_{\ee}K$, or $M\elemsub_{\ee,\Sigma_n}K$ if we also have $M\elemsub_{\Sigma_n}K$.
Strictly speaking, We view second-order structures as two-sorted first-order structures, so by extensions of second-order structures, we mean extensions of the corresponding two-sorted first-order structures. 
We write $(M,\mathcal{X})\elemsub_{\Sigma_n^0}(K,\mathcal{Y})$ if all the $\Sigma_n^0(M,\mathcal{X})$-formulas are absolute between the two structures. We say an extension of second-order structures is an \defm{end extension} if its first-order part is an end extension, and we denote this by $(M,\mathcal{X})\subseteq_{\ee}(K,\mathcal{Y})$, or $(M,\mathcal{X})\elemsub_{\ee,\Sigma_n^0}(K,\mathcal{Y})$ if we also have $(M,\mathcal{X})\elemsub_{\Sigma_n^0}(K,\mathcal{Y})$.

\section{Second-order restricted ultrapowers}
The second-order restricted ultrapower resembles the usual ultrapower construction in model theory, but instead of working on the class of all subsets of the model and all functions from the model to itself, we only consider the sets and functions in the second-order universe. For completeness, we review the definition and some basic facts about it.
All the results appear in \cite{art:Kirby_uf} except Lemma~\ref{lem:groundCA} and Corollary~\ref{cor:1stele}. Throughout this section we fix some arbitrary second-order structure $(M,\mathcal{X})\models \RCA_0$.
\begin{definition}[Second-order restricted ultrapower]
    The second-order part of $(M,\mathcal{X})$ forms a Boolean algebra under inclusion and Boolean operations. Let $\u$ be an ultrafilter on $\mathcal{X}$ such that all the elements of $\u$ are cofinal in $M$, and $\f$ be the class of all the total functions from $M$ to $M$ in $\mathcal{X}$. Define an equivalence relation $\sim$ on $\f$ by 
    \[f\sim g\iff \{i\in M\mid f(i)=g(i)\}\in \u,\]
    where $f,g\in\f$.
    The second-order restricted ultrapower of $M$, denoted by $\ultp$, is the set of equivalence classes $[f]$ for $f\in \f$ modulo $\sim$.
    The interpretations of symbols in the language of first-order arithmetic are defined similarly:
    \[[f]+[g]=[f+g],\]
    \[[f]\times[g]=[f\times g],\]
    \[[f]<[g]\iff \{i\in M\mid f(i)<g(i)\}\in \u.\]
    Here $f+g$ and $f\times g$ are the pointwise sum and product of $f$ and $g$ as functions.
    $M$ naturally embeds into $\ultp$ by identifying elements of $M$ with constant functions.
    Moreover, $\ultp$ is a proper extension of $M$ since the equivalence class of identity function of $M$ is greater than any equivalence class of constant function.
    $\ultp$ admits a natural second-order expansion inherited from $\mathcal{X}$, namely for $A\in\mathcal{X}$ and $[f]\in\ultp$, we define
    \[[f]\in A\iff \{i\in M\mid f(i)\in A\}\in \u.\]
    We denote the expanded structure of the ultrapower by $(\ultp,\mathcal{X})$.
    It is easy to show that for $i\in M$ and $A\in\mathcal{X}$, $i\in A$ holds in $(M,\mathcal{X})$ if and only if it holds in $(\ultp,\mathcal{X})$, so we may view $(\ultp,\mathcal{X})$ as an extension of $(M,\mathcal{X})$, where the second-order part is an injection. 
\end{definition}
The first-order restricted ultrapower is defined similarly, but with $\f$ and $\u$ replaced by the corresponding first-order definable classes.
For example, the $\Delta_1$ ultrapower uses the class of $\Delta_1$-definable subsets and $\Delta_1$-definable total functions.\par 

From now on we also fix an ultrafilter $\u$ on $\mathcal{X}$ such that all the elements of $\u$ are cofinal in $M$.

Generally, \L o\'s's theorem does not hold for restricted ultrapowers, but a restricted version of it does hold:

\begin{thm}[Restricted \L o\'s's theorem]\label{thm:los}
    Let $(\ultp,\mathcal{X})$ be the second-order restricted ultrapower. Then the following holds:
    \begin{enumerate}
        \item If $\phi(\ov x)$ is a $\Sigma_1^0(M,\mathcal{X})$-formula, then 
        \[(\ultp,\mathcal{X})\models \phi(\ov {[f]})\iff \exists {{A}\in{\u}}, A\subseteq \{i\in M\mid (M,\mathcal{X})\models \phi(\ov {f(i)})\}.\]
        \item If $\phi(\ov x)$ is a $\Delta_1^0(M,\mathcal{X})$ formula over $(M,\mathcal{X})$, then 
        \[(\ultp,\mathcal{X})\models \phi(\ov {[f]})\iff \{i\in M\mid (M,\mathcal{X})\models \phi(\ov {f(i)})\}\in\u.\]
    Here the right-hand side makes sense in view of $\Delta_1^0$-comprehension of $(M,\mathcal{X})$.\qed 
    \end{enumerate}
\end{thm}
\begin{cor}~\label{cor:2ndele}
     $(M,\mathcal{X})\elemsub_{\Sigma_2^0}(\ultp,\mathcal{X})$.
\end{cor} 
\begin{proof}
    Let $(M,\mathcal{X})\models \fa x\ex y\psi(x,y)$ for some $\psi(x,y)\in\Delta_0^0(M,\mathcal{X})$. By picking the least witness $y$ of $\psi(x,y)$, we may assume $\psi(x,y)$ defines a total function $f\in\f$, so $(M,\mathcal{X})\models \psi(x,f(x))$ for all $x\in M$. In particular, for each $g\in \f$,
    \[(M,\mathcal{X})\models\psi(g(x),f\circ g(x)).\]
    Here $f\circ g$ is the composition of $f$ and $g$. By restricted \L o\'s's theorem, $(\ultp,\mathcal{X})\models\psi([g],[f\circ g])$ for each $[g]\in\ultp$, so that $(\ultp,\mathcal{X})\models\fa x\ex y\psi(x,y)$.
\end{proof}

\begin{definition}
    We say $\u$ is \defm{additive} if whenever $f\in \f$ is bounded, then there is a $c\in M$ such that $\{i\in M\mid f(i)=c\}\in \u$.
\end{definition}

\begin{lem}\label{lem:additive}
    If $\u$ is additive, then $\ultp$ is an end extension of $M$.
\end{lem}
\begin{proof}
    If $\ultp\models [f]<b$ for some $b\in M$, then we may define 
    \[g(i)=\begin{cases}
			0, & \text{if }f(i)\geq b.\\
            f(i), & \text{if }f(i)<b.
		 \end{cases}\]
$[g]=[f]$ and $g$ is a total bounded function in $\mathcal{X}$. The additiveness of $\u$ implies that there is $c\in M$ such that $\{i\in M\mid g(i)=c\}\in\u$, that is, $\ultp\models [g]=[f]=c$.
\end{proof}
The following lemma and corollary enable us to transfer the comprehension in $(M,\mathcal{X})$ into the ultrapower via $\Sigma_2^0$-elementarity, and reduce the case of $\bd\Sigma_{n+2}$ to $\bd\Sigma_2^0$ uniformly in the construction of our main result.

\begin{lem}~\label{lem:groundCA}
    For each $n\geq 1$, if $(M,\mathcal{X})$ satisfies $\Sigma_n$-comprehension, then each instance of $\Sigma_n$- and $\Pi_n$-comprehension in $(M,\mathcal{X})$ transfers to $(\ultp,\mathcal{X})$. Formally speaking, for any first-order formula $\phi(x)$ in $\Sigma_n(M)$ or $\Pi_n(M)$, if there is some $A\in\mathcal{X}$ such that
    \[(M,\mathcal{X})\models\fa x(x\in A\leftrightarrow \phi(x)),\]
    then $(\ultp,\mathcal{X})\models\fa x(x\in A\leftrightarrow \phi(x))$ as well.
\end{lem}
\begin{proof}
    We prove the statement for all the $\phi(x)$ in $\Sigma_k(M)$ and $\Pi_k(M)$ simultaneously by induction on $k\leq n$.
    For $k=1$, let $\phi(x)$ be any formula in $\Sigma_1(M)$ or $\Pi_1(M)$, then $\fa x(x\in A\leftrightarrow \phi(x))$ is a $\Pi_2^0(M,\mathcal{X})$-formula, so the same holds in $(\ultp,\mathcal{X})$ by Corollary~\ref{cor:2ndele}. For the induction step, suppose the statement holds for all the formulas in $\Sigma_k(M)$ and $\Pi_k(M)$, and take any $\phi(x):=\ex y\psi(x,y)\in\Sigma_{k+1}(M)$ where $\psi(x,y)\in \Pi_k(M)$. By $\Sigma_n$-comprehension in $(M,\mathcal{X})$, there exist $A,B\in \mathcal{X}$ such that
    \[(M,\mathcal{X})\models\fa x(x\in A\leftrightarrow \ex y\psi(x,y)),\]
    \[(M,\mathcal{X})\models\fa {\langle x,y\rangle}(\langle x,y\rangle \in B\leftrightarrow \psi(x,y)).\]
    By induction hypothesis, the second clause transfers to $(\ultp,\mathcal{X})$. We also have the following relation of $A$ and $B$ from the two statements above:
    \[(M,\mathcal{X})\models \fa x(x\in A\leftrightarrow \ex y\langle x,y\rangle \in B).\]
    This fact also transfers to $(\ultp,\mathcal{X})$ by Corollary~\ref{cor:2ndele}. 
    Putting the two statements in $(\ultp,\mathcal{X})$ together, we have
    \[(\ultp,\mathcal{X})\models\fa x(x\in A\leftrightarrow \ex y\psi(x,y)).\]
    The case for $\phi(x)\in\Pi_{k+1}(M)$ is exactly the same. This completes the induction.
\end{proof}

\begin{cor}~\label{cor:1stele}
    If $(M,\mathcal{X})$ satisfies $\Sigma_n$-comprehension for some $n\in \IN$, then $M\elemsub_{\Sigma_{n+2}}\ultp$ as an extension between models of first-order arithmetic.
\end{cor}
\begin{proof}
    Let $M\models \fa x \ex y\psi(x,y)$ for some $\psi(x,y)\in \Pi_n(M)$. Take $A\in\mathcal{X}$ such that 
    \[(M,\mathcal{X})\models\fa {\langle x,y\rangle}(\langle x,y\rangle\in A\leftrightarrow \psi(x,y)),\]
    \[(M,\mathcal{X})\models \fa x\ex y\langle x,y\rangle\in A.\]
    By Lemma~\ref{lem:groundCA} and Corollary~\ref{cor:2ndele}, both formulas hold in $(\ultp,\mathcal{X})$, which implies $(\ultp,\mathcal{X})\models \fa x \ex y\psi(x,y)$.
\end{proof}

\section{Constructions of end extensions}
In this section, we present the constructions of end extensions by the second-order ultrapower.
We first answer Question~\ref{question:kc} affirmatively. In view of Corollary~\ref{cor:1stele}, we only need to deal with the case for the base level in the second-order context.
\begin{thm}\label{thm:main}
    For any countable model $(M,\mathcal{X})\models\bd\Sigma_2^0+\WKL_0$, there is a proper end extension $(M,\mathcal{X})\elemsub_{\ee,\Sigma_2^0}(K,\mathcal{X})\models\bd\Sigma_1^0$.
\end{thm}
We will give two proofs: The first proof is suggested by Tin Lok Wong. We ensure that the second-order ultrapower satisfies $\bd\Sigma_1^0$ by properly embedding it into a coded ultrapower as an initial segment;
the second proof guarantees $\bd\Sigma_1^0$ directly in the construction of the ultrafilter.
$\WKL_0$ plays a central role in both constructions.
\begin{proof}
    Following the construction in~\cite[Theorem 2]{incoll:PKcodedultp}, by an iterated arithmetic completeness theorem within $(M,\mathcal{X})$, there is an end extension $M\subseteq_{\ee}L\models \ind\Delta_0$ such that $\mathcal{X}=\SSy_M(L)$.\par
    We build a coded ultrapower with respect to both $M$ and $L$. Let
    \[\g=\{g\in L\mid g\text{ codes a total function from $M$ to $L$}\},\]
    \[\f=\{f\in L\mid f \text{ codes a total function from $M$ to $M$}\},\]
    and $\u$ be an ultrafilter on $\mathcal{X}$. Following the construction by Paris and Kirby~\cite[Theorem 3]{incoll:PKinitial}, let $\g/\u$ be the coded ultrapower with respect to $\g$ and $\u$.
    Similar to the second-order ultrapower, coded ultapowers also satisfy restricted \L o\'s's theorem for $\Delta_0$-formulas and thus $L\elemsub_{\Delta_0}\g/\u$.
    Moreover, $\g/\u$ is a cofinal extension of $L$, as each element $[g]\in\g/\u$ is bounded by its code $g\in L$.
    So $L\elemsub_{\Sigma_1}\g/\u$ and in particular, $\g/\u\models\ind\Delta_0$.
    Since $\f\subseteq\g$, $\ultp$ is a substructure of $\g/\u$.
    On the other hand, since $\mathcal{X}=\SSy_M(L)$, $\ultp$ is isomorphic to a second-order ultrapower with respect to $\mathcal{X}$ and $\u$, so $(M,\mathcal{X})\elemsub_{\ee,\Sigma_2^0}(\ultp,\mathcal{X})$.
    
    We want to pick a sufficiently generic ultrafilter $\u$, such that both $M$ and $\ultp$ are proper initial segments of $\mathcal{G}/\u$.
    We construct $\u$ in countably many stages. For each $k\in\IN$ we construct some $A_k\in\mathcal{X}$ that is cofinal in $M$, and $A_k\supseteq A_{k+1}$.
    We enumerate all the pairs $\langle f,g\rangle$ such that $f\in \f$ and $g\in\mathcal{G}$ as $\{\langle f_k,g_k\rangle\}_{k\in\IN}$, and all the functions in $\f$ bounded by some $b\in M$ as $\{h_k\}_{k\in\IN}$.\par
    \textbf{Stage $\boldsymbol{0}$}: Let $A_0=M$.\par
    \textbf{Stage $\boldsymbol{2k+1}$ }($\ultp\subseteq_\ee\g/\u$): Consider 
    \[A=A_{2k}\cap\{x\in M\mid L\models g_k(x)<f_k(x)\}.\]
    Since $L\models\ind\Delta_0$, $A\in\SSy_M(L)=\mathcal{X}$.
    If $A$ is cofinal in $M$, then let $A_{2k+1}=A$. Otherwise let $A_{2k+1}=A_{2k}$ and move on to the next stage.\par
    \textbf{Stage $\boldsymbol{2k+2}$ }($M\subseteq_\ee\ultp$): Assume $h_k$ is bounded by $b\in M$.
    Then
    \[(M,\mathcal{X})\models\excf x\exlt y b(x\in A_{2k+1}\wedge h_k(x)=y).\]
    Since $(M,\mathcal{X})\models\bd\Sigma_2^0$, there is some $c<b$ such that 
    \[\{x\in M\mid (M,\mathcal{X})\models x\in A_{2k+1}\wedge h_k(x)=c\}\]
    is cofinal in $M$. Let this set be $A_{2k+2}$ and move on to the next stage.\par
    Finally, let $\u=\{A\in \mathcal{X}\mid \exin k {\IN} A_{k}\subseteq A\}$. It is not hard to see that $\u$ is an ultrafilter and each element of $\u$ is cofinal in $M$. This completes the construction of $\u$.\par
    \textbf{Verification}: We verify that $M\subseteq_{\ee}\ultp\subseteq_{\ee}\g/\u$. For $\ultp\subseteq_{\ee}\g/\u$, fix any $g\in\g$.
    For each $f\in \f$, take $k\in\IN$ such that $\langle f_k,g_k\rangle=\langle f,g\rangle$.
    At stage $2k$, if 
    \[A=A_{2k}\cap\{x\in M\mid g(x)<f(x)\}\]
    is cofinal in $M$, then $\{x\in M\mid g(x)<f(x)\}\in\u$.
    Let $\hat{g}\in\mathcal{G}$ be defined by 
    \[\hat{g}(x)=\begin{cases}
			g(x), & \text{if }g(x)< f(x).\\
            0, & \text{if }g(x)\geq f(x).
		 \end{cases}\]
   Then $[\hat{g}]\in \ultp$ and $[g]=[\hat{g}]$ in $\g/\u$. Otherwise, if $A$ is bounded in $M$, then we are forced to have $\{x\in M\mid g(x)\geq f_k(x)\}\in \u$. By the restricted \L o\'s's theorem in $\g/\u$, $\g/\u\models [g]>[f_k]$. So $\ultp$ is an proper initial segment of $\g/\u$, which implies $(\ultp,\SSy_{\ultp}(\g/\u))\models\bd\Sigma_1^0$.
   
   For $M\subseteq_{\ee}\ultp$, suppose $h\in\f$ is bounded. Then, take $k\in\IN$ such that $h=h_k$, and at stage $2k+2$ the choice of $A_{2k+2}$ forces $\{x\in M\mid (M,\mathcal{X})\models h(x)=c\}\in\u$ for some $c\in M$. So $\u$ is additive with respect to $\f$, and $M\subseteq_{\ee}\ultp$ by Lemma~\ref{lem:additive}.\par
   For each $A\in\mathcal{X}$, let $a\in L$ be the element that codes $A\subseteq M$. By the restricted \L o\'s's theorem for $\Delta_0$-formulas in both $\ultp$ and $\g/\u$, it is not hard to prove that for each $f\in\f$, 
    \[(\ultp,\mathcal{X})\models [f]\in A\iff(\g/\u,\mathcal{X})\models [f]\in a.\]
    So we may embed the second-order part $\mathcal{X}$ of $(\ultp,\mathcal{X})$ into $\SSy_{\ultp}(\g/\u)$.
    Since $(\ultp,\SSy_{\ultp}(\g/\u))\models\bd\Sigma_1^0$, we have $(\ultp,\mathcal{X})\models\bd\Sigma_1^0$.
\end{proof}
\begin{thm}\label{thm:kaufmann-clote}
    For each $n\in \IN$ and any countable model $M\models\bd\Sigma_{n+2}$, there is a $\Sigma_{n+2}$-elementary proper end extension $M\elemsub_{\ee,\Sigma_{n+2}}K\models \bd\Sigma_{n+1}$.
\end{thm}
\begin{proof}
    We first expand $M$ to a second-order structure satisfying $\bd\Sigma_2^0$ by adding all the $\Sigma_n$-definable sets into the second-order universe, then we further $\omega$-extend it to some countable $(M,\mathcal{X})\models\bd\Sigma_2^0+\WKL_0$. By Theorem~\ref{thm:main}, there is an ultrapower $(M,\mathcal{X})\elemsub_{\ee,\Sigma_2^0}(\ultp,\mathcal{X})$ which satisfies $\bd\Sigma_1^0$. Since all the $\Sigma_n$-definable sets of $M$ are in $\mathcal{X}$, $(M,\mathcal{X})$ satisfies $\Sigma_n$-comprehension and $M\elemsub_{\ee,\Sigma_{n+2}}\ultp$ by Corollary~\ref{cor:1stele}.
    
    For $\ultp\models\bd\Sigma_{n+1}$, suppose $\ultp\models \falt x {[g]}\ex y\theta(x,y,[f])$ for some $[g]\in\ultp$ and $\theta\in\Pi_n$ where $[f]\in\ultp$ is the only parameter in $\theta$.
    By $\Pi_n$-comprehension in $M$, let $A\in \mathcal{X}$ such that 
    \[(M,\mathcal{X})\models \fa {\langle x,y,z\rangle} (\langle x,y,z\rangle \in A\leftrightarrow\theta(x,y,z)).\]
     $(\ultp,\mathcal{X})$ satisfies the same formula by Lemma~\ref{lem:groundCA}, and thus
    \[(\ultp,\mathcal{X})\models\falt x {[g]}\ex y\langle x,y,[f]\rangle\in A.\]
    By $\bd\Sigma_1^0$ in $(\ultp,\mathcal{X})$,
    \[(\ultp,\mathcal{X})\models \ex b\falt x {[g]}\exlt y b\langle x,y,[f]\rangle \in A,\]which means $\ultp\models\ex b\falt x {[g]}\exlt y b\theta(x,y)$, so $\ultp\models\bd\Sigma_{n+1}$.
\end{proof}
Yet simple, this construction does not reveal a syntactical proof of $\bd\Sigma_{n+2}\vdash\wrd\Pi_{n+1}$. We make $(\ultp,\mathcal{X})\models\bd\Sigma_1^0$ by embedding it into a larger ultrapower $\g/\u$ as an initial segment, and the core argument is wrapped inside the construction of $\g/\u$.\par
Here we present another construction that directly guarantees each instance of $\bd\Sigma_1^0$ in the ultrapower, and hence provide more insights. The construction relies on a simple yet powerful lemma resulting from $\WKL_0$, which states that if a $\Pi_1^0$-definable bounded multi-valued function is total, then we may select a single-valued choice function of it within the second-order universe.
The lemma also leads to a syntactical proof of $\bd\Sigma_{n+2}\vdash \wrd\Pi_{n+1}$.
\begin{lem}~\label{lem:RPi_1^-}
    Fix a model $(M,\mathcal{X})\models\WKL_0$. Let $\theta(x,y,z)\in \Delta_0^0(M,\mathcal{X})$. If $(M,\mathcal{X})\models \fa x \exlt {y}{f(x)}\fa z\theta(x,y,z)$ for some total function $f\in \mathcal{X}$, then there is a total function $P\in \mathcal{X}$ such that:
        \[(M,\mathcal{X})\models\fa x(P(x)<f(x)\wedge\fa z\theta(x,P(x),z)).\]
\end{lem}
\begin{proof}
    Consider the following tree $T$ which is $\Delta_1^0$-definable in $(M,\mathcal{X})$:
    \[\sigma\in T\iff \falt{x,z}{\len\sigma}(\sigma(x)<f(x)\wedge\theta(x,\sigma(x).z))\]
Obviously $T$ is bounded by $f$, so we only need to show that $T$ is infinite.
Let $F(x)=\max_{x'< x}f(x')$. 
For any $x\in M$, by $\ind\Sigma_1^0$, let $\sigma_x\in M$ be a coded sequence of length $x$ such that for any $x'<x$ and $y'<F(x)$,
\[\sigma_x(x')=y'\iff (M,\mathcal{X})\models\fa z\theta(x',y',z)\wedge\falt w {y'}\neg\fa z\theta(x',w,z).\]
Then we have
\[(M,\mathcal{X})\models\falt{x'}{x}\fa z(\sigma_x(x')<f(x)\wedge \theta(x',\sigma_x(x'),z)).\] 
This means $\sigma_x$ is an element of $T$ of length $x$, so $T$ is infinite.
By $\WKL_0$, take an infinite path $P$ of $T$. Clearly $P$ satisfies the two requirements above.
\end{proof}


\begin{proof}[Alternative proof of Theorem~\ref{thm:main}]
    We construct an ultrafilter $\u$ on $\mathcal{X}$ in $\omega$ many stages. 
    Along the construction, we gradually guarantee that the ultrapower $(\ultp,\mathcal{X})\models \bd\Sigma_1^0$ and $\u$ is additive.\par
    Enumerate all the triples $\{\langle \ex z \theta_k(x,y,z),f_k, g_k\rangle\}_{k\in \IN}$, where $\theta_k(x,y,z)\in\Delta_0^0(M,\mathcal{X})$ and $f_k,g_k$ are total functions in $\mathcal{X}$, and enumerate all the bounded total functions in $\mathcal{X}$ as $\{h_k\}_{k\in \IN}$. For each $k\in\IN$, at stage $k$ we construct a cofinal set $A_k\in\mathcal{X}$ such that $A_k\supseteq A_{k+1}$ for all $k\in \IN$, and the resulting ultrafilter $\u=\{A\in\mathcal{X}\mid\exin {k}{\IN} A\supseteq A_k\}$.\par
    \textbf{Stage $\boldsymbol{0}$ }: Set $A_0=M\in\mathcal{X}$.\par
    \textbf{Stage $\boldsymbol{2k+1}$ }($(\ultp,\mathcal{X})\models \bd\Sigma_1^0$):
    At these stages we want to guarantee the following instances of $\bd\Sigma_1^0$:
    \[\falt{y}{[g_k]}\ex z\theta_k([f_k],y,z)\rightarrow\ex b\falt{y}{[g_k]}\exlt {z}{b}\theta_k([f_k],y,z).\]
    The general idea is that we first try to `force' the consequent of above implication to be true in $(\ultp,\mathcal{X})$. If we succeed, then the entire statement is true. Otherwise, we apply Lemma~\ref{lem:RPi_1^-} to argue that the antecedent is already guaranteed to be false in the ultrapower.\par
    Consider the $\Sigma_1^0$-definable set 
    \[A=A_{2k}\cap\{x\in M\mid \ex b\falt {y}{g_k(x)}\exlt{z}{b}\theta_k(f_k(x),y,z)\}.\]
    If it is cofinal in $M$, then there is a cofinal subset $A^*\in\mathcal{X}$ of $A$ by H\'ajek--Pudl\'ak~\cite[I, Theorem 3.22]{book:HP}. Let $A_{2k+1}=A^*$ and proceed to stage $2k+2$. If $A$ is not cofinal in $M$, we let $A_{2k+1}=A_{2k}$ and proceed directly to stage $2k+2$.\par
    \textbf{Stage $\boldsymbol{2k+2}$ }(Additiveness of $\u$): This part is exactly the same as the construction of stage $2k+2$ in the first proof of Theorem~\ref{thm:main}.
    
    Finally, let $\u=\{A\in\mathcal{X}\mid \exin k {\IN}A\supseteq A_k\}$. It is not hard to see that $\u$ is an ultrafilter, and each element of $\u$ is cofinal in $M$. This completes the whole construction. \par
    \textbf{Verification}: Let $(\ultp,\mathcal{X})$ be the corresponding second-order restricted ultrapower. We show that $(M,\mathcal{X})\elemsub_{\ee,\Sigma_2^0}(\ultp,\mathcal{X})$ and $(\ultp,\mathcal{X})\models\bd\Sigma_1^0$. The elementarity is given by Corollary~\ref{cor:2ndele}. $(M,\mathcal{X})\subseteq_\ee(\ultp,\mathcal{X})$ follows from the exact same reasoning as in the first proof of Theorem~\ref{thm:main}. For $(\ultp,\mathcal{X})\models \bd\Sigma_1^0$, consider an arbitrary instance of $\bd\Sigma_1^0$ in $(\ultp,\mathcal{X})$:
    \[\falt{y}{[g]}\ex z\theta([f],y,z)\rightarrow\ex b\falt{y}{[g]}\exlt {z}{b}\theta([f],y,z),\]
    where $\theta\in \Delta_0^0(M,\mathcal{X})$. Here without loss of generality, we assume there is only one first-order parameter $[f]\in \ultp$ in $\theta$. 
    Assume at stage $2k+1$, we enumerate $\langle \ex z \theta_k(x,y,z),f_k, g_k\rangle=\langle \ex z\theta(x,y,z),f,g\rangle$, and $A_{2k}\in \mathcal{X}$ is the cofinal subset of $M$ we obtained from the previous stage.
    Suppose we are in the first case of the construction in this stage, i.e.,
    \[A=A_{2k}\cap \{x\in M\mid \ex b\falt{y}{g(x)}\exlt {z}{b}\theta(f(x),y,z)\}\]
    is cofinal in $M$, then by the construction there exists some $A^*\subseteq A$ in $\u$. By restricted \L o\'s's theorem, $(\ultp,\mathcal{X})\models \ex b\falt{y}{[g]}\exlt{z}{b}\theta([f],y,z)$, so the instance of $\bd\Sigma_1^0$ is true. If we are in the second case, assuming $A$ is bounded by $d\in M$, then
    \[(M,\mathcal{X})\models\falg x d(x\in A_{2k}\rightarrow \fa b\exlt y {g(x)}\falt z b\neg\theta(f(x),y,z)).\]
    By $\bd\Sigma_1^0$ in $M$, this fact is equivalent to: 
    \[(M,\mathcal{X})\models\falg{x}{d}\exlt{y}{g(x)}(x\in A_{2k}\rightarrow\fa z\neg\theta(f(x),y,z)).\]
    By Lemma~\ref{lem:RPi_1^-}, there is a total function $P\in \mathcal{X}$ such that 
    \begin{enumerate}
        \item $(M,\mathcal{X})\models\falg {x}{d} P(x)<g(x).$
        \item $(M,\mathcal{X})\models\falg {x}{d} (x\in A_{2k}\rightarrow\fa z\neg \theta(f(x),P(x),z) ).$ 
    \end{enumerate}
    The first clause implies $(\ultp,\mathcal{X})\models [P]<[g]$ by restricted \L o\'s's theorem. Suppose $(\ultp,\mathcal{X})\models \ex z\theta([f],[P],z)$, then by restricted \L o\'s's theorem for $\Sigma_1^0$ formulas, there is some ${A'}\in \u$ such that ${A'}\subseteq \{x\in M\mid \ex z\theta(f(x),P(x),z)\}$, but by the second clause, $A'\cap A_{2k}$ is bounded by $d$, which contradicts $A'\cap A_{2k}\in\u$. So $(\ultp,\mathcal{X})\models \fa z\neg\theta([f],[P],z)$, and the instance of $\bd\Sigma_1^0$ considered is vacuously true.
\end{proof}

We now proceed to the construction of end extensions for characterizing countable models of $\ind\Sigma_{n+2}$. We first define $K\models M\hyp\ind\Sigma_{n+1}$ for an end extension $M\elemsub_\ee K$, and introduce some equivalent definitions of it.
\begin{definition}
    For each $n\in\IN$, let $M,K$ be models of $\ind\Delta_0+\exp$ and $M\subseteq_\ee K$.
        We say $K\models M\hyp\ind\Sigma_{n+1}$ if for any $\phi(x)\in\Sigma_{n+1}(K)$ and $a\in M$,
        \[K\models \phi(0)\wedge\falt x a(\phi(x)\rightarrow\phi(x+1))\rightarrow\falt x a\phi(x).\]
    Notice that we allow parameters in $K$ in $\phi$ while the bound $a$ must be in $M$.
\end{definition}
\begin{lem}\label{lem:tfaeMind}
    For each $n\in\IN$, let $M$ be a model of $\ind\Delta_0+\exp$, $K\models\ind\Sigma_n$ and $M\subseteq_\ee K$, then the following are equivalent:
    \begin{enumerate}\tfaeenum
        \item \label{lem:equivMind:1}
        $K\models M\hyp\ind\Sigma_{n+1}$.
        \item \label{lem:equivMind:2}
        For any $\phi(x)\in\Sigma_{n+1}(K)$ and $a\in M$,
        \[K\models \ex c\falt {x} a(\phi(x)\leftrightarrow x\in c).\]
        \item \label{lem:equivMind:3}
        For any $\theta(x,y)\in\Pi_{n}(K)$ and $a\in M$,
        \[K\models \ex b\falt x a(\ex y\phi(x,y)\leftrightarrow\exlt y b\theta(x,y)).\]
    \end{enumerate}
\end{lem}
\begin{proof}
    We show $\clref{lem:equivMind:1}\Leftrightarrow\clref{lem:equivMind:2}$ and $\clref{lem:equivMind:2}\Leftrightarrow\clref{lem:equivMind:3}$. 
    For $\clref{lem:equivMind:1}\Rightarrow\clref{lem:equivMind:2}$, first by modifying a standard argument, one can show that $\clref{lem:equivMind:1}$ implies the least number principle for $\Pi_{n+1}$-formulas that is satisfied by some element of $M$.
    Then we pick the least $c<2^a\in M$ such that 
    \[K\models \falt x a(\phi(x)\rightarrow x\in c).\]
    Such $c$ will code $\phi(x)$ for $x<a$ by the minimality of $c$.
    For $\clref{lem:equivMind:2}\Rightarrow\clref{lem:equivMind:1}$, take the code $c$ of $\phi(x)$ below $a\in M$. Then, one can prove the instance of $M\hyp\ind\Sigma_{n+1}$ for $\phi(x)$ by replacing $\phi(x)$ with $x\in c$ and applying $K\models\ind\Delta_0$.
    For $\clref{lem:equivMind:2}\Rightarrow\clref{lem:equivMind:3}$, take the code $c$ of $\{x<a\mid K\models\ex y\theta(x,y)\}$ by $\clref{lem:equivMind:2}$. Consider the $\Sigma_{n+1}$-formula
    \[\Phi(v):=\ex b\falt x v(x\in c\leftrightarrow\exlt y b\theta(x,y)).\]
    It is not hard to show that $K\models\Phi(0)\wedge\fa v(\Phi(v)\rightarrow\Phi(v+1))$.
    By $M\hyp\ind\Sigma_{n+1}$ (from $\clref{lem:equivMind:2}\Rightarrow\clref{lem:equivMind:1}$) we have $K\models \Phi(a)$, which implies $\clref{lem:equivMind:3}$. For $\clref{lem:equivMind:3}\Rightarrow\clref{lem:equivMind:2}$, let $\phi(x):=\ex y\theta(x,y)$ for some $\theta\in \Pi_n(K)$. 
    By $\clref{lem:equivMind:3}$, there is some $b\in K$ such that
    \[K\models\falt x a(\ex y\theta(x,y)\leftrightarrow\exlt y b\theta(x,y)).\]
    By $\ind\Sigma_n$ in $K$, there is some $c\in K$ that codes $\{x<a\mid K\models\exlt y b\theta(x,y)\}$. Such $c$ will serve as the witness for $\clref{lem:equivMind:2}$.
\end{proof}
\begin{thm}\label{thm:main-char}
    For any countable model $(M,\mathcal{X})\models \ind\Sigma_2^0$, there is a proper end extension $(M,\mathcal{X})\elemsub_{\ee,\Sigma_2^0}(K,\mathcal{X})$ such that for any $\Sigma_1^0(K,\mathcal{X})$-formula $\phi(z)$ and $a\in M$, the set $\{z<a\mid (K,\mathcal{X})\models\phi(z)\}$ is coded in $K$.
\end{thm}
\begin{proof}
    The construction is a mild generalization of the construction in~\cite[Proposition 7]{incoll:clote}. 
    We make the ultrapower $(\ultp,\mathcal{X})$ satisfy the coding requirement by maximizing each $\Sigma_1^0$-definable subset of $\ultp$ that is bounded by some element of $M$.

    Enumerate all the pairs $\{\langle \ex y\theta_k(x,y,z),a_k\rangle\}_{k\in\IN}$ where $\theta_k(x,y,z)\in\Delta_0^0(M,\mathcal{X})$ and $a_k\in M$, and enumerate all the bounded total functions in $\mathcal{X}$ as $\{h_k\}_{k\in\IN}$.
    At each stage $k$ we construct a cofinal set $A_k\in\mathcal{X}$ such that $A_k\supseteq A_{k+1}$ for all $k\in \IN$, and the resulting ultrafilter $\u=\{A\in\mathcal{X}\mid\exin {k}{\IN} A\supseteq A_k\}$.
    
    \textbf{Stage $\boldsymbol{0}$}: Let $A_0=M$.\par
    \textbf{Stage $\boldsymbol{2k+1}$ }(coding $\Sigma_1^0$ sets): Consider the following $\Pi_2^0$-formula where $c<2^{a_k}$:
    \[\excf x(x\in A_{2k}\wedge \fain z c\ex y\theta_k(x,y,z)).\]
    By $(M,\mathcal{X})\models\ind\Sigma_2^0$, there exists a maximal $c_0<2^{a_k}$ satisfying the formula above.
    Similar to the alternative proof of Theorem~\ref{thm:main}, let $A_{2k+1}\in \mathcal{X}$ be a cofinal subset of the following $\Sigma_1^0$-definable subset of $M$:
    \[\{x\in M\mid (M,\mathcal{X})\models x\in A_{2k}\wedge \fain z {c_0}\ex y\theta_k(x,y,z)\}.\]\par
    \textbf{Stage $\boldsymbol{2k+2}$ }(Additiveness of $\u$): This part is exactly the same as stage $2k+2$ in the proof of Theorem~\ref{thm:main}.
    
    Finally, let $\u=\{A\in\mathcal{X}\mid \exin k {\IN}A\supseteq A_k\}$. It is not hard to see that $\u$ is an ultrafilter, and each element of $\u$ is cofinal in $M$. This completes the whole construction.
    
    \textbf{Verification}: Let $(\ultp,\mathcal{X})$ be the corresponding second-order restricted ultrapower. $(M,\mathcal{X})\elemsub_{\ee,\Sigma_2^0}(\ultp,\mathcal{X})$ follows in exactly the same way as in Theorem~\ref{thm:main}.

    For the coding requirement of $(\ultp,\mathcal{X})$, consider any $\Sigma_1^0$-formula $\ex y\theta([f],y,z)$ where $\theta\in \Delta_0^0(\ultp,\mathcal{X})$ and $[f]\in\ultp$ is the only first-order parameter of $\theta$.
    For any $a\in M$, assume for some $k\in\IN$, at stage $2k+1$ we enumerate the pair $\langle \ex y\theta(f(x),y,z),a\rangle$. We show that the maximal $c_0\in M$ we obtained in the construction codes $\{z<a\mid (\ultp,\mathcal{X})\models \ex y\theta([f],y,z)\}$.

    For each $z<a$ such that $z\in c_0$, since $A_{2k+1}$ is a subset of $\{x\in M\mid (M,\mathcal{X})\models\ex y\theta(f(x),y,z)\}$, $(\ultp,\mathcal{X})\models\ex y\theta([f],y,z)$ by restricted \L o\'s's theorem for $\Sigma_1^0$-formulas.
    On the other hand, for each $z'<a$ such that $z'\notin c_0$, if $(\ultp,\mathcal{X})\models\ex y\theta([f],y,z')$, then by restricted \L o\'s's theorem for $\Sigma_1^0$-formula again there is some $B\in \u$ such that
    \[B\subseteq \{x\in M\mid (M,\mathcal{X})\models\ex y\theta(f(x),y,z')\}.\]
    Since $B\cap A_{2k+1}\in \u$ is still cofinal in $M$, we have
    \[\excf x(x\in A_{2k}\wedge \fain z {{c_0}\cup\{z'\}}\ex y\theta(f(x),y,z)),\]
    which contradicts the maximality of $c_0$ in the construction.
    So $(\ultp,\mathcal{X})\models\neg\ex y\theta([f],y,z')$.
\end{proof}
\begin{thm}\label{thm:clote}
    For each $n\in \IN$ and any countable model $M\models\ind\Sigma_{n+2}$, there is an $\Sigma_{n+2}$-elementary proper end extension $M\elemsub_{\ee,\Sigma_{n+2}}K\models M\hyp\ind\Sigma_{n+1}$.
\end{thm}
\begin{proof}
    The proof is mostly the same as Theorem~\ref{thm:kaufmann-clote}. 
    We expand $M$ to a second-order structure $(M,\mathcal{X})$ satisfying $\ind\Sigma_2^0+\RCA_0$ by adding all the $\Delta_{n+1}$-definable subsets of $M$ into the second-order universe.
    Such $(M,\mathcal{X})$ satisfies $\Sigma_n$-comprehension.
    By Theorem~\ref{thm:main-char}, there exists a second-order ultrapower extension $(M,\mathcal{X})\elemsub_{\ee,\Sigma_2^0}(\ultp,\mathcal{X})$ that codes all the $\Sigma_1^0$-definable subset bounded by some element of $M$. Since all the $\Sigma_n$-definable sets of $M$ are in $\mathcal{X}$, $(M,\mathcal{X})$ satisfies $\Sigma_n$-comprehension and $M\elemsub_{\ee,\Sigma_{n+2}}\ultp$ by Corollary~\ref{cor:1stele}. 
    
    For $\ultp\models M \hyp\ind \Sigma_{n+1}$, we only need to show that $\ultp$ satisfies condition $\clref{lem:equivMind:3}$ in Lemma~\ref{lem:tfaeMind}. 
    Let $\phi(x):=\ex y\theta(x,y,[f])$ be any $\Sigma_{n+1}$-formula, where $\theta\in \Pi_n$ and $[f]\in \ultp$ is the only parameter in $\theta$, and $a$ be some element of $M$. Since $(M,\mathcal{X})$ satisfies $\Sigma_n$-comprehension, there is some $A\in\mathcal{X}$ such that 
    \[(M,\mathcal{X})\models\fa {\langle x,y,z\rangle} (\langle x,y,z\rangle \in A\leftrightarrow \theta(x,y,z)).\]
    The same formula holds in $(\ultp,\mathcal{X})$ by Lemma~\ref{lem:groundCA}, and thus
    \[(\ultp,\mathcal{X})\models\fa x(\ex y\langle x,y,[f]\rangle \in A\leftrightarrow\phi(x)).\]
    There is some $c\in K$ that codes $\{x<a\mid (\ultp,\mathcal{X})\models\ex y\langle x,y,[f]\rangle \in A\}$.
    Such $c$ will also code $\{x<a\mid\ultp\models\phi(x)\}$.  
\end{proof}

\section{The weak regularity principle}
In the final section we provide an application of Lemma~\ref{lem:RPi_1^-}: We introduce the weak regularity principle $\wrd\phi$, a variant of the regularity principle, and determine its strength within the I-B hierarchy. One of the main application is to prove the converse of Theorem~\ref{thm:clote}, and give a model-theoretic characterization of countable models of $\ind\Sigma_{n+2}$ analogous to Theorem~\ref{thm:paris-kirby}.

Mills and Paris~\cite{art:Paris-Mills} introduced the regularity principle $\rd \phi$ to be the universal closure of the following formula:
\[\excf x \exlt {y}{a}\phi(x,y)\rightarrow\exlt y a \excf x\phi(x,y).\]
For any formula class $\Gamma$, let
\[\rd\Gamma=\ind\Delta_0\cup\{\rd\phi\mid\phi\in\Gamma\}.\]
It is also shown in~\cite{art:Paris-Mills} that $\rd\Pi_n\Leftrightarrow\rd\Sigma_{n+1}\Leftrightarrow\bd\Sigma_{n+2}$ for each $n\in \IN$.
The weak regularity principle is defined by replacing the $\exists^{\cf} x$ by $\forall x$ in the antecedent of implication in $\rd\phi$. 
\begin{definition}
    Let $\phi(x,y)$ be a formula in first-order arithmetic with possibly hidden variables. The \defm{weak regularity principle} $\wrd\phi$ denotes the universal closure of the following formula:
    \[\fa x \exlt {y}{a}\phi(x,y)\rightarrow\exlt{y}{a}\excf x\phi(x,y).\]
    For any formula class $\Gamma$, define 
    \[\wrd\Gamma=\ind\Delta_0\cup\{\wrd\phi\mid\phi\in\Gamma\}.\]
\end{definition}
The strength of the weak regularity principle behaves more complicated within the arithmetic hierarchy compared to the regularity principle.
The principle for most of the natural class of formulas correspond to collection schemes, whereas induction schemes are only equivalent to the principle for a highly restricted subclass of $\Sigma_0(\Sigma_{n+1})$-formulas.  
We will show that for each $n\in\IN$, $\wrd\Sigma_0(\Sigma_{n})$, $\wrd\Sigma_{n+1}$, $\wrd\Pi_{n+1}$ and $\wrd(\Sigma_{n+1}\vee\Pi_{n+1})$ are all equivalent to $\bd\Sigma_{n+2}$,
and $\wrd(\Sigma_{n+1}\wedge\Pi_{n+1})$ is equivalent to $\ind\Sigma_{n+2}$.

The weak regularity principle may also be viewed as an infinitary version of the pigeonhole principle, and similar phenomenon arises with the strength of pigeonhole principle in the I-B hierarchy.
Dimitracopoulos and Paris~\cite{art:dcphp} proved that $\php\Sigma_{n+1}$ and $\php\Pi_{n+1}$ are equivalent to $\bd\Sigma_{n+1}$.
$\php(\Sigma_{n+1}\vee\Pi_{n+1})$ and $\php\Sigma_0(\Sigma_{n+1})$ are equivalent to $\ind\Sigma_{n+1}$.

\begin{lem}\label{lem:B->rpi}
     For each $n\in\IN$, $\bd\Sigma_{n+2}\vdash\wrd\Pi_{n+1}$.
\end{lem}
\begin{proof}
    We only show the case of $n=0$ and the rest can be done by relativizing to the $\Sigma_n$ universal set. Let $M\models\bd\Sigma_{2}$, we first expand $M$ to a second-order structure $(M,\mathcal{X})\models \bd\Sigma_2^0+\WKL_0$.
    Let $\phi(x,y):=\fa z \theta(x,y,z)$ for some $\theta\in \Delta_0$.  
    Suppose $M\models\fa x \exlt {y}{a}\phi(x,y)$. By applying Lemma~\ref{lem:RPi_1^-} with $\theta$ and $f(x) \equiv a$ as a constant function, we obtain some total function $P\in \mathcal{X}$ such that $(M,\mathcal{X})\models \fa x(P(x)<a\wedge\fa z\theta(x,P(x),z))$. By $\bd\Sigma_2^0$, there is some $y_0<a$ such that there are infinitely many $x$ satisfying $P(x)=y_0$, which implies $M\models\exlt y a\excf x \phi(x,y)$. So $M\models \wrd\Pi_1$.
\end{proof} 
\begin{cor}\label{cor:B->rpivsigma}
    For each $n\in\IN$, $\bd\Sigma_{n+2}\vdash \wrd{(\Sigma_{n+1}\vee\Pi_{n+1})}$.
\end{cor}
\begin{proof}
    Fix $n\in\IN$ and let $M\models \bd\Sigma_{n+2}$, $\phi(x,y)\in\Sigma_{n+1}(M)$ and $\psi(x,y)\in\Pi_{n+1}(M)$. Suppose $M\models \fa x\exlt y a(\phi(x,y)\vee\psi(x,y))$ for some $a\in M$.
    If $M\models\falg x b\exlt y a\psi(x,y)$ for some $b\in M$, then by Lemma~\ref{lem:B->rpi}, $M\models\excf x\exlt y a\psi(x,y)$ and the conclusion holds.
    Otherwise $M\models\excf x\exlt y a\phi(x,y)$, then by $\rd\Sigma_{n+1}$, $M\models \exlt y a\excf x\phi(x,y)$ and the conclusion holds again.
\end{proof}
\begin{remark}
    There is also a direct model-theoretic proof similar to Proposition~\ref{prop:endextrpi}. One only need to notice that over $\bd\Sigma_{n+1}$, $\fa x\exlt y a(\phi(x,y)\vee\psi(x,y))$ is equivalent to a $\Pi_{n+2}$-formula.
\end{remark}
\begin{lem}\label{lem:rpi->B2}
    For each $n\in\IN$, $\wrd(\Sigma_n\wedge\Pi_n)\vdash\ind\Sigma_{n+1}$.
\end{lem}
\begin{proof}
    We prove for each $k\leq n$, $\ind\Sigma_k+\wrd(\Sigma_{n}\wedge\Pi_n)\vdash \ind\Sigma_{k+1}$, and the lemma follows by induction.
    Let $M\models \ind\Sigma_{k}+\wrd(\Sigma_{n}\wedge\Pi_n)$, suppose $M\models \neg\ind\Sigma_{k+1}$, then there is a proper cut $I\subseteq M$ defined by $\phi(y):=\ex x\theta(x,y)$, where $\theta(x,y)\in\Pi_k(M)$. Let $\mu(x,y):=\falt {y'} y\exlt {x'} x\theta(x',y')$. Then $\mu(x,y)\in\Pi_k(M)$ over $\ind\Sigma_k$ and $\phi^*(y)=\ex x\mu(x,y)$ also defines a cut $J\subseteq I$.
    $J$ is closed under successor by its definition.
    For any $x\in M$, if $M\models\mu(x,y)$ for all $y\in J$, then $y\in J$ is defined by $\mu(x,y)$ in $M$, which contradicts $M\models \ind\Sigma_k$. So for each $x\in M$, we may take the largest $y\in J$ satisfying $\mu(x,y)$ by $\ind\Sigma_k$. Fixing some arbitrary $a>J$, we have 
    \[M\models \fa x\exlt y a(\mu(x,y)\wedge\neg\mu(x,y+1)).\]
    Applying $\wrd(\Sigma_n\wedge\Pi_n)$, there is some $y_0<a$ such that
    \[M\models \excf x(\mu(x,y_0)\wedge\neg\mu(x,y_0+1)).\]
    By the definition of $\mu(x,y)$, this implies $y_0\in J$ and $y_0+1\notin J$, which contradicts the fact that $J$ is a proper cut of $M$ closed under successor.
    So $M\models \ind\Sigma_{k+1}$.
\end{proof}

\begin{lem}[Independently by Leszek A. Ko\l odziejczyk]\label{lem:rpi->B1}
    For each $n\in\IN$, $\wrd{\Sigma_0(\Sigma_n)}\vdash\bd\Sigma_{n+2}$.
\end{lem}
\begin{proof}
    Let $M\models\wrd\Sigma_0(\Sigma_{n})$. We show $M\models \rd\Pi_n$, which is equivalent to $\bd\Sigma_{n+2}$. 
    By Lemma~\ref{lem:rpi->B2}, $M\models\ind\Sigma_{n+1}$.
    Suppose $M\models \excf x \exlt y a \phi(x,y)$ for some $\phi\in\Pi_n(M)$, and without loss of generality, we assume $M\models \exlt y a\phi(0,y)$.
    For each $z\in M$, we find the largest $x<z$ such that $M\models\exlt y a\phi(x,y)$, and `color' $z$ by the witness $y$. Formally,
    \[M\models\fa z\exlt y a\exlt x z(\phi(x,y)\wedge\fain {x'}{(x,z)}\neg\exlt y a\phi(x',y)),\]
    where $(x,z)$ refers to the open interval between $x$ and $z$.
    By $\wrd\Sigma_0(\Sigma_n)$, 
    \[M\models\exlt y a\excf z\exlt x z(\phi(x,y)\wedge\fain {x'}{(x,z)}\neg\exlt y a\phi(x',y)).\]
    which implies $M\models \exlt y a\excf x\phi(x,y)$.
\end{proof}

\begin{thm}\label{thm:rpiequivalence}
    For each $n\in \IN$, $\wrd\Sigma_0(\Sigma_n)\Leftrightarrow\wrd(\Sigma_{n+1}\vee\Pi_{n+1})\Leftrightarrow\bd\Sigma_{n+2}$.
\end{thm}
\begin{proof}
    $\wrd\Sigma_0(\Sigma_n)\vdash\bd\Sigma_{n+2}$ follows from Lemma~\ref{lem:rpi->B1}. $\bd\Sigma_{n+2}\vdash\wrd(\Sigma_{n+1}\vee\Pi_{n+1})$ follows from Corollary~\ref{cor:B->rpivsigma}. $\wrd(\Sigma_{n+1}\vee\Pi_{n+1})\vdash\wrd\Sigma_0(\Sigma_n)$ is trivial.
\end{proof}
There is an analog of Proposition~\ref{prop:endextrpi} for $M\elemsub_{\ee,\Sigma_{n+2}}K\models M\hyp\ind\Sigma_{n+1}$ and $\wrd(\Sigma_{n+1}\wedge\Pi_{n+1})$. It also leads to a model-theoretic proof of $\ind\Sigma_{n+2}\vdash \wrd(\Sigma_{n+1}\wedge\Pi_{n+1})$.
\begin{proposition}\label{prop:ind->rdmodel}
    Let $M\models\ind\Delta_0+\exp$. For each $n\in\IN$, if there is a proper $\Sigma_{n+2}$-elementary end extension $M\elemsub_{\ee,\Sigma_{n+2}}K\models M\hyp\ind\Sigma_{n+1}$, then $M\models\wrd(\Sigma_{n+1}\wedge\Pi_{n+1})$. 
\end{proposition}
\begin{proof}
    Let $\theta(x,y,z)\in\Sigma_{n}(M)$, $\sigma(x,y,w)\in\Pi_{n}(M)$ and $\phi(x,y):=\fa z\theta(x,y,z)\wedge \ex w\sigma(x,y,w)\in\Sigma_{n+1}\wedge\Pi_{n+1}$. Suppose $M\models \fa x\exlt y a\phi(x,y)$ for some $a\in M$. 
    Over $\ind\Sigma_{n+1}$, it is equivalent to the following $\Pi_{n+2}$ formula 
    \[\fa x \fa b\exlt y a(\falt z b\theta(x,y,z)\wedge\ex w\sigma(x,y,w)).\]
    So both $M$ and $K$ satisfy the formula above by elementarity. Pick some arbitrary $d\in K\setminus M$, then
    \[K\models \fa b\exlt y a(\falt z b\theta(d,y,z)\wedge \ex w\sigma(d,y,w)).\]
    By Lemma~\ref{lem:tfaeMind}$\clref{lem:equivMind:2}$, there is some $b\in K$ such that 
    \[K\models \falt y a(\fa z\theta(d,y,z)\leftrightarrow\falt z b\theta(d,y,z)),\]
    which implies
    \[K\models \exlt y a(\fa z\theta(d,y,z)\wedge\ex w\sigma(d,y,w)).\]
    Pick a witness $c<a$ in $M$ such that $K\models \fa z\theta(d,c,z)\wedge\ex w\sigma(d,c,w)$, i.e., $K\models\phi(d,c)$. Now for each $b\in M$, $K\models \exlg x b\phi(x,c)$, where this is witnessed by $d$. Transferring each of these formulas to $M$, we have $M\models\exlg x b\phi(x,c)$ for any $b\in M$, i.e., $M\models\excf x\phi(x,c)$.
\end{proof}

\begin{thm}\label{thm:rdwedge<->ind}
    For each $n\in\IN$, $\wrd(\Sigma_{n+1}\wedge\Pi_{n+1})\Leftrightarrow\ind\Sigma_{n+2}$.
\end{thm}
\begin{proof}
    $\wrd(\Sigma_{n+1}\wedge\Pi_{n+1})\vdash\ind\Sigma_{n+2}$ follows from Lemma~\ref{lem:rpi->B2}. For the other direction, given any countable model $M\models\ind\Sigma_{n+2}$, there is a proper end extension $M\elemsub_{\ee,\Sigma_{n+2}}K\models M\hyp\ind\Sigma_{n+1}$ by Theorem~\ref{thm:clote}, and then $M\models\wrd(\Sigma_{n+1}\wedge\Pi_{n+1})$ by Proposition~\ref{prop:ind->rdmodel}.
\end{proof}
\begin{remark}
    In H\'ajek-Pudl\'ak~\cite[I, Lemma 2.49]{book:HP} it was shown that every $\Sigma_0(\Sigma_{n+1})$-formula has the following normal form:
    \[\Qklt {1}{u_1} {v_1}\dots \Qklt {k}{u_k} {v_k}\Psi(u_1\dots u_k,v_1\dots v_k,w_1\dots w_l). \]
    where $k,l\in \IN$, each $Q_k$ is either $\forall$ or $\exists$, $\Psi$ is a Boolean combination of $\Sigma_{n+1}$-formulas and the variable sets $\{u_i\}_{i\leq k}$, $\{v_i\}_{i\leq k}$ and $\{w_i\}_{i\leq l}$ are pairwise disjoint.  

    Our proof of Proposition~\ref{prop:ind->rdmodel} can be refined to show that $\ind\Sigma_{n+2}\vdash\wrd\phi$, where $\phi(x,y)\in\Sigma_0(\Sigma_{n+1})$, and if written in the normal form above, $x$ does not appear in $\{v_i\}_{i\leq k}$, i.e., $x$ is not permitted to appear in the bound of any bounded quantifiers. 
    In contrast, the instance $\phi(z,y)$ of $\wrd\Sigma_0(\Sigma_{n})$ we used to prove $\bd\Sigma_{n+2}$ in Lemma~\ref{lem:rpi->B1} starts with $\exists x<z$ explicitly.
\end{remark}
Finally, we establish the characterization of countable models of $\ind\Sigma_{n+2}$ as promised.
\begin{thm}\label{thm:paris-kirby-ind}
    Let $M$ be a countable model of $\ind\Delta_0$. For each $n\in\IN$, $M$ satisfies $\ind\Sigma_{n+2}$ if and only if $M$ admits a proper $\Sigma_{n+2}$-elementary end extension $K\models M\hyp\ind\Sigma_{n+1}$.
\end{thm}
\begin{proof}
    The direction from left to right follows by Theorem~\ref{thm:paris-kirby-ind}. For the other direction, if $M\elemsub_{\ee,\Sigma_{n+2}}K\models M\hyp\ind\Sigma_{n+1}$, then $M\models\wrd(\Sigma_{n+1}\wedge\Pi_{n+1})$ by Proposition~\ref{prop:ind->rdmodel}, and thus $M\models\ind\Sigma_{n+2}$ by Theorem~\ref{thm:rdwedge<->ind}. 
\end{proof}

The main remaining problem now is to find a purely syntactic proof of Theorem~\ref{thm:rdwedge<->ind}. We conjecture that a more refined tree construction similar to the approach in Lemma~\ref{lem:RPi_1^-} would solve the problem.
\begin{problem}
   Give a direct proof of $\ind\Sigma_{n+2}\vdash\wrd(\Sigma_{n+1}\wedge\Pi_{n+1})$ (and also $\ind\Sigma_{n+2}\vdash\wrd\phi$ where $\phi(x,y)\in\Sigma_0(\Sigma_{n+1})$ as described in the remark above) without using end extensions.
\end{problem}
\section*{Acknowledgement}
The author's research was partially supported by the Singapore Ministry of Education Tier 2 grant AcRF MOE-000538-00 as well as by the NUS Tier 1 grants AcRF R146-000-337-114 and R252-000-C17-114. This work is contained in the author's Ph.D. thesis. I would like to thank my two supervisors Tin Lok Wong and Yue Yang for their guidance, encouragement and helpful discussions. 
I am also indebted to Leszek A. Ko\l odziejczyk for carefully reading drafts of the paper and providing many helpful suggestions.
\bibliographystyle{plain}
\bibliography{end}

\end{document}